\documentclass{amsart}
\usepackage{amscd,pb-diagram,bbm}
\newtheorem{thm}{Theorem}
\newtheorem{lem}[thm]{Lemma}
\newtheorem{prop}[thm]{Proposition}
\newtheorem{cor}[thm]{Corollary}

\theoremstyle{definition}
\newtheorem{defn}{Definition}
\theoremstyle{remark}
\newtheorem{rem}{Remark}
\newtheorem{exmp}{Example}[section]
\newcommand\motx{\mathcal O_{\tilde X}}
\newcommand\tx{{\tilde X}}
\newcommand\mo{\mathcal O}
\newcommand\lra{\longrightarrow}
\newcommand\JF{\operatorname{Jac}F}
\newcommand\CC{\mathbbm k}
\newcommand\PP{\mathbb P}
\newcommand\tp{\tilde{{\mathbb P}^{4}}}
\DeclareMathOperator{\Sing}{Sing}
\title[Hodge numbers of hypersurfaces with ordinary triple points]{Hodge
  numbers of hypersurfaces in $\PP^{4}$ with ordinary triple points}
\author{S\l awomir Cynk}
\address{Institute of Mathematics,
Jagiellonian University,
ul. {\L}ojasiewicza 6,
30-348 Krak\'ow,
Poland}\email{slawomir.cynk@uj.edu.pl}

\thanks{Research partially supported by the National Science Center grant no. 2014/13/B/ST1/00133}
\subjclass[2010]{Primary: {14J30};  Secondary {14J17}}
\keywords{Hodge number, triple points, defect}

\begin{document}

\begin{abstract}
We give a formula for the Hodge numbers of hypersurfaces  in $\PP^{4}$
with ordinary triple points. 
 \end{abstract}
\maketitle
\section*{Introduction}

Let $X$ be a degree $d$ hypersurface in the projective four--space
$\PP^{4}$ with ordinary triple   points as the only
singularities. Denote by $\Sigma=\Sing(X)$  the  singular locus 
of $X$ and by $\tilde X$
the natural crepant resolution  of singularities of $X$ by the
blow--up of the singular locus $\Sigma\subset\PP^{4}$.

The main goal of the present paper it to proof the following formula
for the Hodge numbers of $\tilde X$  (Cor.~\ref{cor:main})
  \[
   h^{1,1}(\tilde X)=1+\mu+\delta,\qquad
 h^{1,2}(\tilde X)=h^{1,2}(X_{smooth})-11\mu+\delta,
 \]
where $X_{\rm smooth}$ is a smooth hypersurface in $\PP^{4}$ of the
same degree, $\mu=\#\Sigma$ is the number of singularities of $X$ and
$\delta$ is non--negative integer called the defect. 
Our main result is an extension of  Werner's defect formula for nodal
hypersurfaces (\cite{Werner}) as explained in Remark~\ref{rem:fin}. 

Three dimensional varieties with ordinary triple points admit
crepant divisorial resolution of singularities, 
they are very suitable for explicit constructions. Probably the main
obstacle in applications was the lack of formulas for the invariants
(Betti numbers or Hodge numbers). 
Our formula can be also used (via a triple cyclic covering,
cf.~Rem.~\ref{rem:wps}) to study surface of degree divisible by  three with
ordinary triple points as  the only singularities, surfaces of a small
degree with triple points have been studied in
\cite{stevens,EPS}. 

The proof of the main theorem is based on a study of exact sequences
of cohomology groups of differential forms (with logarithmic poles),
cf. \cite{CR}. The main difference is that in the present paper it is not
sufficient to determine the dimensions of various cohomology groups
but we need to study the rank of the natural map
$H^{1}(\Omega^{3}_{\tp}(\tx)\lra H^{1}(\Omega^{3}_{\tx}(\tx)))$ which
was possible because of an explicit represantation of basis of
various cohomology group (based on f.i. \cite{pet-st}). 
 
\section{Logarithmic differential forms}

Throughout the paper $X$ is a degree $d$
hypersurface  in the projective space $\PP^{4}$ with ordinary triple
points as the only singularities. Let
$\Sigma:=\{P_{1},\dots,P_{\mu}\}$ be the singular locus of $X$ and
$F\in \CC[X_{0},\dots,X_{4}]$ the homogeneous equation of $X$. 
The strict transform $\tx$ of $X$ under the blow--up
$\sigma:\PP^{4}\lra \tp$ is a crepant resolution of singularities. 
Let $E_{i}:=\sigma^{-1}(P_{i})$ be the exceptional divisor of $\sigma$
over a singular point $P_{i}$ and denote
$E:=\sigma^{-1}(\Sigma)=E_{1}+\dots+E_{\mu}$.
In this situation the following formulas hold (cf. \cite{CR})
\begin{prop}\leavevmode\label{prop:intro}
  \begin{enumerate}
  \item $\sigma_{\ast}\mo_{\tp}(-mE)\cong \mathcal J_{m\Sigma},\
    \text{ for }m\ge0$,
  \item $R^{i}\sigma_{\ast}\mo_{\tp}(-mE)=0$, for $i\not=0, m\ge0$,
  \item $H^{i}(\motx)=0$, for $i=1,2$,
  \item $H^{1}(\Omega^{3}_{\tp})=0$ for $i\le2$,
  \item $H^{i}(\Omega^{4}_{\tp}(\tx))\cong
    H^{i}(\Omega^{4}_{\PP^{4}}(X))$,
  \item $H^{i}(\Omega^{4}_{\tp}(2\tx))\cong
    H^{i}(\Omega^{4}_{\PP^{4}}(2X)\otimes \mathcal J_{3\Sigma})$.
   \end{enumerate}
\end{prop}

\begin{lem}\label{lem:h12}
\(\displaystyle h^{1,2}(\tx)=h^{0}(\Omega^{3}_{\tx}(\tx))-h^{0}(\Omega^{3}_{\tp}(\tx))
  + \dim\operatorname{Ker}(H^{1}\Omega^{3}_{\tp}(\tx)\lra
H^{1}\Omega^{3}_{\tx}(\tx)) 
\)
\end{lem}
\begin{proof}
  By Serre duality, $H^{0}\Omega^{2}_{\tx}=H^{2}(\motx)=0$.
  Since $\tilde X$ is smooth  the following
  sequence is exact (\cite[2.3(b)]{esnview})
    \[0\lra \Omega^{3}_{\tp} \lra \Omega^{3}_{\tp}(\log\tx)\lra
\Omega^{2}_{\tx}\lra0\]
which implies  
  \begin{eqnarray*}
    &&H^{0}\Omega^{2}_{\tx}\cong H^{0}\Omega^{3}_{\tp}(\log\tx)=0\\
    &&H^{1}\Omega^{2}_{\tx}\cong H^{1}\Omega^{3}_{\tp}(\log\tx).
  \end{eqnarray*}

Similarly, the following sequence is  exact (\cite[2.3(c)]{esnview})
  \[0\lra \Omega^{3}_{\tp}(\log\tx) \lra \Omega^{3}_{\tp}(\tx)\lra
    \Omega^{3}_{\tx}(\tx)\lra0.\]
 and the derived long exact sequence
  \[0\lra H^{0}\Omega^{3}_{\tp}(\tx)\lra H^{0}\Omega^{3}_{\tx}(\tx)\lra
H^{1}\Omega^{3}_{\tp}(\log\tx) \lra H^{1}\Omega^{3}_{\tp}(\tx)\lra
H^{1}\Omega^{3}_{\tx}(\tx) \]
implies the assertion of the lemma.
\end{proof}
\begin{cor} The following sequence is exact
  \begin{equation}
    \label{eq:cd1}
H^{0}\Omega^{3}_{\PP^{4}}(X) \lra
H^{0}\left(\Omega^{3}_{\PP^{4}}(X)\otimes\mo_{\Sigma}\right)
\lra
H^{1}\left(\Omega^{3}_{\tp}(\tx) \right)
\lra 0.    
  \end{equation}

\end{cor}
\begin{proof}
  By direct computations
  $\sigma^{*}\Omega^{3}_{\PP^{4}}=\Omega^{3}_{\tp}(\log E)(-3E)$, so
  we have an exact sequence (\cite[2.3(c)]{esnview})
  \[0\lra\sigma^{*}\Omega^{3}_{\PP^{4}}(X)\otimes\mo_{\tp}(-E) \lra
  \Omega^{3}_{\tp}(\tx) \lra \Omega^{3}_{E}(3)\lra0,
  \]
  applying the direct image yields   
  \[\sigma_{*}\Omega^{3}_{\tp}(\tx)=\Omega^{3}_{\tp}(X)\otimes\mathcal
       J_{\Sigma},\qquad
       R^{i}\sigma_{*}\Omega^{3}_{\tp}(\tx)=0 \quad\text{ for }i>0.
     \]
     From the Leray spectral sequence we get
  \[H^{i}\Omega^{3}_{\tp}(\tx)=H^{i}(\Omega^{3}_{\PP^{4}}(X)\otimes\mathcal
  J_{\Sigma})\]
and the assertion follows from the cohomology derived sequence
associated to
\[ 0 \lra \Omega^{3}_{\PP^{4}}(X)\otimes\mathcal
  J_{\Sigma}\lra \Omega^{3}_{\PP^{4}}(X) \lra \Omega^{3}_{\PP^{4}}(X)\otimes
  \mo_{\Sigma} \lra 0.
\]
\end{proof}

\begin{cor}
The following sequence is exact
\begin{equation}\label{eq:cd2}
    H^{0}(\Omega^{4}_{\PP^{4}}(2X)) \lra
    H^{0}(\Omega^{4}_{\PP^{4}}(2X)\otimes \mo_{3\Sigma}) \lra
    H^{1}(\Omega^{3}_{\tx}(\tx)) \lra0
\end{equation}
\end{cor}

\begin{proof}
  By the adjunction formula $\Omega^{3}_{\tx}(\tx)
  \cong\Omega^{4}_{\tp}(2\tx)|\tx$, so we have an exact sequence
  \[0\lra\Omega^{4}_{\tp}(\tx) \lra \Omega^{4}_{\tp}(2\tx) \lra
    \Omega^{3}_{\tx}(\tx) \lra 0,\]
  the derived long exact sequence and the Proposition~\ref{prop:intro} yield 
\[H^{1}(\Omega^{3}_{\tx}(\tx))\cong
  H^{1}(\Omega^{4}_{\PP^{4}}(2X)\otimes \mathcal J_{3\Sigma}).
\]
The assertion follows now from the exact sequence 
\[0\lra \Omega^{4}_{\PP^{4}}(2X)\otimes \mathcal J_{3\Sigma} \lra
  \Omega^{4}_{\PP^{4}}(2X) \lra
  \Omega^{4}_{\PP^{4}}(2X)\otimes \mathcal \mo_{3\Sigma} \lra 0.
  \]
\end{proof}

\section{Main result}

We keep the notation introduced in the previous section. 
\begin{defn}
  Define the \emph{equisingular ideal of $X$} as 
  \[I_{\rm eq}:=\bigcap_{i=1}^{\mu}(m_{i}^{3}+\JF),\]
where $m_{i}$ is the (maximal) ideal of $P_{i}$ and $\JF$ is the jacobian
ideal of $F$.
\end{defn}

Let
$S=\bigoplus\limits_{d=0}^{\infty}S_{d}=\CC[X_{0},\dots,X_{4}]$ be the
graded ring of polynomials in five variables, for a homogeneous ideal $I\subset
S$ we denote by $I^{(d)}:=I\cap S^{d}$ the degree $d$ graded summand
of $I$.

\begin{thm}
\label{t:main}
  \begin{eqnarray*}
   &&h^{1,1}(\tilde X)=\dim (I_{\rm eq}^{(2d-5)})-\binom {2d-1}4+12\mu+1\\
 &&h^{1,2}(\tilde X)=\dim (I_{\rm eq}^{(2d-5)})-5\binom d4
 \end{eqnarray*}
\end{thm}
\begin{proof}
  Consider the following commutative diagram with exact rows
  \eqref{eq:cd1} and   \eqref{eq:cd2} 
\begin{equation*}
  \begin{diagram}\dgARROWLENGTH=1.2em
\node{(S^{d-4})^{\oplus5}/S^{d-5}}\arrow{s,r}{\cong} \node{\CC^{4\mu}}\arrow{s,r}{\cong}\\
\node{H^{0}(\Omega^{3}_{\PP^{4}}(X))} \arrow{s,r}{\xi}\arrow{e,t}{\theta} 
\node{H^{0}\left(\Omega^{3}_{\PP^{4}}(X)\otimes\mo_{\Sigma}\right)}
\arrow{s,r}{\beta}\arrow{e,t}{\alpha}
\node{H^{1}\left(\Omega^{3}_{\tp}(\tx) \right)}
\arrow{s,r}{\Phi}\arrow{e} \node{0}\\
    \node{H^{0}(\Omega^{4}_{\PP^{4}}(2X))} \arrow {e,t}{\eta}
    \node{H^{0}(\Omega^{4}_{\PP^{4}}(2X)\otimes \mo_{3\Sigma})} \arrow{e,t}{\gamma}
    \node{H^{1}(\Omega^{3}_{\tx}(\tx))} \arrow{e}\node0\\
\node{S^{2d-5}}\arrow{n,r}{\cong}\node{\CC^{15\mu}}\arrow{n,r}{\cong}
  \end{diagram}
\end{equation*}
We shall describe explicitly all maps in the above diagram. Denote by $K_{j}$  the contraction with the vector field
$\frac{\partial}{\partial X_{j}}$ and by $\Omega$ the $4$--form
$\sum_{i=0}^{4}(-1)^{i}dX_{0}\wedge\dots\wedge \widehat{dX_{i}}
\wedge\dots\wedge dX_{4}$.
The two vertical isomorphisms in the first column are given by
\[
(S^{d-4})^{\oplus5}\ni(A_{0},\dots,A_{4})\longmapsto \sum_{i=0}^{k}\frac{A_{i}}FK_{i}\Omega
\in H^{0}(\Omega^{3}_{\PP^{4}}(X))
\]
(with the inclusion $S^{d-5}\ni A\longmapsto
(AX_{0},\dots,AX_{4})\in (S^{d-4})^{\oplus5}$).
and
\[
  S^{(2d-5)}\ni A\longmapsto \frac A{F^{2}}\Omega\in 
  H^{0}(\Omega^{4}_{\PP^{4}}(2X))
\]
In terms of these isomorphisms the homomorphism $\eta$ associates to a
degree $2d-5$ homogeneous polynomial its 
3--jets at the singular points $P_{1},\dots,P_{\mu}$, while $\theta$
associates to a quintuple $A_{0},\dots,A_{4}$ the values at 
singular points $P_{j}$ in the vector space $(\CC^{4})^{\mu}$
identified with $\bigoplus\limits_{j=1}^{\mu}(\CC^{5}/P_{j}\CC)$.
Finally $\xi$ is the exterior derivative $\omega\mapsto d\omega$ so
\[d\left(\sum_{i=0}^{4}\frac{A_{i}}FK_{i}\Omega\right) = 
\frac1{F^{2}}\sum_{i=0}^{4}(F\frac{\partial A_{i}}{\partial X_{i}} -
A_{i}\frac{\partial F}{\partial X_{i}}) \Omega\]
and consequently 
$\beta(A_{0}^{1},\dots,A_{4}^{1},\dots,A_{0}^{\mu},\dots,A_{4}^{\mu})$
is given by 3--jets of $\sum_{i=0}^{4}A_{i}^{j}(P_{j})\frac{\partial F}{\partial
X_{i}}$ at singular points $P_{1},\dots, P_{\mu}$. 
As the polynomial $F$ has an ordinary triple point at $P_{j}$ partial derivatives
of $F$ at $P_{j}$ are linearly independent modulo the third power $m_{j}^{3}$
of the maximal ideal $m_{j}$, and consequently the map $\beta$ is
injective. In particular $\dim\operatorname{Im}(\beta)=4\mu$. 
Diagram chasing with simple linear algebra yields
\[\dim\operatorname{Ker}\Phi=
h^{1}(\Omega^{3}_{\tp}(\tx)) - \dim(\operatorname{Im}\beta)
+\dim(\operatorname{Im}\eta\cap\operatorname{Im}\beta)
.\]
The local description shows that 
$\operatorname{Im}(\eta)\cap\operatorname{Im}(\beta)\cong (I_{\rm
  eq}/\bigcap _{j=1}^{\mu} m_{j}^{3})^{(2d-5)}$.
Using the   Lemma~\ref{lem:h12} we get formula for $h^{1,2}$.

Observe that the Milnor number at
an ordinary triple point is $16$, resolution replaces this point with a smooth
cubic surface with the Euler number 9, finally $e(\tilde
X)=-d^{4}+5d^{3}-10d^{2}+10d+24\mu$.  
As the resolution of $X$ is crepant we have $h^{0,3}(\tilde
X)=\binom{d-1}4$ and the formula for $h^{1,1}$ follows.
\end{proof}
Since the point  $P_{i}$ ($i=1,\dots,\mu$) is an ordinary triple point, the codimension of the
ideal $(m_{i}^{3}+\JF)$ equals 11, consequently the expected dimension
of $I_{\rm eq}^{(2d-5)}$ is $\binom{2d-1}5-11\mu$. We shall call 
the difference between ``the actual dimension'' and ``the expected
dimension'' the defect. 
\begin{defn}
  Define the \emph{defect} of the hypersurface $X$ as the 
  integer
  \[\delta :=\dim(I_{\rm eq}^{(2d-5)})-\left(\binom{2d-1}4-11\mu\right).\]
\end{defn}
\begin{cor}\label{cor:main}
  \[
   h^{1,1}(\tilde X)=1+\mu+\delta,\qquad
 h^{1,2}(\tilde X)=h^{1,2}(X_{smooth})-11\mu+\delta
 \]
A hypersurface $X$ is $\mathbb Q$--factorial iff it has no defect.
\end{cor}
\begin{rem}\label{rem:fin}
  The above definition of defect is a direct generalization the definition of
  the defect of a hypersurface with A--D--E singularities in
  \cite[Def.~2.1]{CR}. 
\end{rem}
\begin{rem}\label{rem:wps}
  Our main theorem generalizes (with the same proof) to the case of a
  degree $d$ hypersurface $X$ in a weighted projective space
  $\PP(w_{0},\dots,w_{4})$ with ordinary triple points provided
  $X\cap\Sing (\PP(w_{0},\dots,w_{4}))=\emptyset$. The last condition
  implies in particular that the weights $w_{i}$ are pairwise
  co--prime and  divide the degree $d$.
  
  In this situation we define the \emph{defect} as
  \[\delta:=\dim(I_{eq}^{(2d-|w|)})-(\dim S^{2d-|w|}-11\mu)\]
  and the same arguments (using \cite{dolg}) yield the following
  weighted version of the main theorem
  \begin{eqnarray*}
   &&h^{1,1}(\tilde X)=\dim (I_{\rm eq}^{(2d-|w|)})-\dim
      S^{2d-|w|}+12\mu+1=1+\mu+\delta\\ 
 &&h^{1,2}(\tilde X)=\dim (I_{\rm eq}^{(2d-|w|)})-\sum_{i=0}^{k}\dim
    S^{d+w_{i}-|w|} = h^{1,2}(X_{smooth})-11\mu+\delta
 \end{eqnarray*}

  The most important example is a triple solid, if $D\subset \PP^{3}$
  is a surface in projective three space of degree divisible by three
  than there exists a triple cyclic cover $\pi:X\lra \PP^{3}$
  branched along $D$. The singularities of $X$ corresponds
  one--to--one to the singularities of $D$, in particular an ordinary
  triple point on $D$ gives an ordinary triple point on $X$. The
  threefold $X$ is given in $\PP(1,1,1,1,d/3)$ by an equation
  $x_{4}^{3}=g(x_{0},\dots,x_{3})$, where $g$ is an equation of $D$.  
  We get a formula analogous to the Clemens defect formula for double
  solid (cf. \cite{Clemens}) with defect defined by the degree
  $\frac53d-4$ component of the equisingular ideal.
\end{rem}

\begin{exmp}
  We shall study a degree six surface in $\PP^{3}$ with ten
  ordinary triple points constructed  in \cite{stevens} as an element
  of a three dimensional family. Let
  \begin{eqnarray*}
    K_1 &=& 2x_1^2-(\varepsilon+2)x_3+\varepsilon^2x_1x_3\;,\\
    K_2 &=& -x_2^2+2\varepsilon x_1+x_2+\varepsilon^2 x_1x_2\;,\\
    K_3 &=& 2x_3^2-2\varepsilon^2 x_2+(6\varepsilon+2)x_3+4\varepsilon^2 x_2x_3\;,\\
    Q   &=& -(\varepsilon+2-x_1)(\varepsilon-x_2)(\varepsilon^2+x_3) +
            x_1(x_2-1)(x_3+3\varepsilon+1)\;. 
\end{eqnarray*}
where $\varepsilon$ is a third root of unity. Then the degree six polynomial
  \begin{eqnarray*}    
    27K_{1}K_{2}K_{3}+2Q^{3}=0
  \end{eqnarray*}
  defines an element of three dimensional family of sextic surfaces
  with ten (maximal possible) number of ordinary triple
  points. Moreover for $p = 67$ with $\varepsilon =-30$ all singular
  points are defined over the base field (\cite{stevens}). Computations conducted with 
  \texttt{Singular} yield $\dim(I_{\rm eq}^{(6)})=30,\, \mu=10,\; \delta=10,\; h^{1,1}=21,\,
  h^{1,2}=3$.   
\end{exmp}

\emph{Acknowledgments}.
I would like to thank Remke Kloosterman for helpful discussions on
this topic.

\end{document}